\documentclass[reqno,english]{amsart}
\usepackage{amsfonts,amsmath,latexsym,verbatim,amscd,mathrsfs,color,array}
\usepackage[colorlinks=true]{hyperref}

\usepackage{amsmath,amssymb,amsthm,amsfonts,graphicx,color}
\usepackage{amssymb}
\usepackage{pdfsync}
\usepackage{epstopdf}
\usepackage[colorlinks=true]{hyperref}

\makeatletter
\def\@currentlabel{2.1}\label{e:dispaa}
\def\@currentlabel{2.21}\label{e:dispau}
\def\@currentlabel{2.22}\label{e:dispav}
\def\@currentlabel{2.23}\label{e:dispaw}
\def\@currentlabel{2.24}\label{e:dispax}
\def\theequation{\thesection.\@arabic\c@equation}
\makeatother
 
\numberwithin{equation}{section}

\newcommand{\rem}{{Rm}}

\newcommand{\bla}{{\lambda,\lambda',h,h'}}
\newcommand{\blaa}{{\bar \lambda,\bar \lambda',\bar h,\bar h'}}
\newcommand{\blaaa}{{\lambda,  \lambda',\bar h,\bar h'}}
\newcommand{\gal}{{\gamma_\lambda}}
\newcommand{\gall}{{\gamma_{\lambda'}}}
\newcommand{\gabll}{{\gamma_{{\bar \lambda}'}}}
\newcommand{\gabl}{{\gamma_{\bar \lambda}}}

\newcommand{\R} {\mathbb R}
\newcommand{\cuad}{{\sqcap\kern-.68em\sqcup}}

\newcommand{\p}{\frac{n+2}{n-2}}

\newcommand{\mmm}{\frac{4}{n-2}}

\newcommand{\be}{\begin{equation}}
\newcommand{\bee}{\begin{equation*}}
\newcommand{\ee}{\end{equation}}
\newcommand{\eee}{\end{equation*}}
\newcommand{\bs}{\begin{split}}
\newcommand{\esp}{\end{split}}

\newcommand{\la}{\lambda}

\renewcommand{\theequation}{\thesection.\arabic{equation}}
 \newtheorem{lem}{Lemma}[section]
\newtheorem{defn}{Definition}[section]
\newtheorem{thm}{Theorem}[section]
\newtheorem{prop}{Proposition}[section]

\newtheorem{remark}{Remark}[section]

\newtheorem{step}{Step}

\newcommand{\bremark}{\begin{remark} \em}
\newcommand{\eremark}{\end{remark} }

\title{ Type I ancient compact solutions of  the \\ Yamabe flow}

\author{Panagiota Daskalopoulos}
\address{ {\bf P. Daskalopoulos:} Department of Mathematics, Columbia University, 2990 Broadway, New York, NY 10027, USA.}
\email{pdaskalo@math.columbia.edu}

\author{Manuel del Pino}
\address{ {\bf M. del Pino:} Departamento de Ingenier\'{\i}a Matem\'atica and CMM, Universidad de Chile, Casilla 170 Correo 3, Santiago, Chile. } \email{delpino@dim.uchile.cl
}

\author{John King}
\address{ {\bf J. King:}  School of Mathematical Sciences, The University of Nottingham, University Park,
Nottingham, NG7 2RD}\email{john.king@nottingham.ac.uk
}

\author{Natasa Sesum}
\address{ {\bf N. Sesum:} Department of Mathematics, Rutgers University, 110 Frelinghuysen road, Piscataway,  NJ 08854, USA.}\email{natasas@math.rutgers.edu
}

\begin{document}

\maketitle

\begin{abstract}
We construct new  ancient compact solutions to the  Yamabe flow.  Our solutions
are rotationally symmetric and converge, as $t \to -\infty$, to  two self-similar complete non-compact  solutions to the Yamabe flow 
moving  in opposite directions. They are type I ancient solutions.

\end{abstract}

\section{Introduction}\label{sec-into}

Let $(M,g_0)$ be a compact manifold without boundary  of dimension $n \geq 3$. If $g = \bar u^{\frac 4{n-2}} \, g_0$
is a metric conformal to $g_0$, the scalar curvature $R$  of $g$ is given in terms of the
scalar curvature $R_0$ of $g_0$ by
$$R= \bar u^{-\frac {n+2}{n-2}} \, \big ( - \bar c_n \Delta_{g_0} \bar u  + R_0 \, \bar u \big )$$
where $\Delta_{g_0}$ denotes the Laplace Beltrami operator with respect to $g_0$ and $\bar c_N =  4 (n-1)/(n-2)$.

In 1989 R. Hamilton introduced the  {\em Yamabe  flow}
\begin{equation}
\label{eq-YF}
\frac{\partial g}{\partial t} = -R\, g
\end{equation}
as an approach to solve the {\em Yamabe problem}  on manifolds of positive conformal Yamabe invariant.
It is the negative $L^2$-gradient flow of the total scalar curvature, restricted to a given conformal class. This was shown by S. Brendle \cite{S2, S1} (up to a technical condition in dim $n \geq 6$).  Significant earlier works  in this directions include those by R. Hamilton \cite{H}, B. Chow \cite{Ch}, R. Ye \cite{Y},
H. Schwetlick and M. Struwe \cite{SS} among many others. The Yamabe conjecture, was previously  shown by R. Shoen 
via elliptic methods in his seminal work \cite{S}.

%
%
%

\smallskip 
In the special case where the background manifold $M_0$ is the  sphere $S^n$ and $g_0$ is the standard
spherical metric $g_{_{S^n}}$, the Yamabe flow evolving a metric $g= \bar u^{\frac 4{n-2}}(\cdot,t)
 \, g_{_{S^n}}$ takes (after  rescaling in time by a constant)
	the form of the {\em  fast diffusion equation}
\begin{equation}
\label{eq-YFS}
(\bar u^\frac{n+2}{n-2})_t  = \Delta_{S^n} \bar u -c_n \bar u,  \qquad c_n = \frac{n(n-2)}{4}.
\end{equation}
Starting with any smooth metric $g_0$ on $S^n$, it follows by the results in \cite{Ch}, \cite{Y} and \cite{DS} that
the solution of \eqref{eq-YFS} with initial data $g_0$ will become singular at some finite time $t < T$ and
$v$ becomes spherical at time $T$,  which means that after
a normalization, the normalized flow converges to the spherical metric.  In addition, $v$ becomes extinct  at $T$.

A metric $g =  \bar u^{\frac 4{n-2}} \, g_{_{S^n}}$ may also be expressed as a metric on $\R^n$ via  stereographic
projection. It follows that if $g =  \hat u^{\frac 4{n-2}} (\cdot,t) \, g_{_{\R^n}}$ (where $g_{_{\R^n}}$ denotes the standard metric
on $\R^n$)  evolves by the Yamabe flow \eqref{eq-YF},  then $\hat  u$ satisfies (after a rescaling in time) the fast diffusion equation on $\R^n$
\begin{equation}
\label{eq-ufd}
(\hat u^p)_t = \Delta \hat u, \quad \qquad p:= \frac{n+2}{n-2}.
\end{equation}
Observe that if $g =  \hat u^{\frac 4{n-2}} (\cdot,t) \, g_{_{\R^n}}$ represents a smooth solution when lifted to $S^n$,
then $\hat{u}(\cdot,t)$ satisfies the growth condition
$$\hat u(y,t) = O  (|y|^{-(n-2)}), \qquad \mbox{as} \,\, |y| \to \infty.$$

\medskip

\begin{defn}[Type I and type II ancient solutions]\label{defn-ancient}
The solution $g = u(\cdot, t)^{\frac 4{n-2}}\, g_0$ to \eqref{eq-YF}  is called ancient if it exists for all time $t\in (-\infty, T)$,
where $T < \infty$.   We will say that the ancient solution $g$ is
of type I,  if its  Riemannian curvature  satisfies
$$\limsup_{t\to-\infty} \, ( |t| \, \max_{M_0} |\mbox{\em Rm}| \ \, (\cdot,t))< \infty.$$
An ancient  solution which is not of type I,  will be called of type
II.

\end{defn}

\smallskip

\noindent The simplest example  of  an ancient solution  to the Yamabe  flow on $S^n$ is 
the {\em  contracting spheres}. 
They are  special solutions  $\bar u$   of \eqref{eq-YFS} which depend only on time $t$ and satisfy the ODE
$$\frac {d \bar u^{\p}}{dt} =-c_n\, \bar u.$$ They are given by
 \begin{equation}
\label{eq-spheres}
\bar u_{_{S}}(p,t) = \left ( \mmm \, c_n\,  (T-t) \right )^{\frac {n-2}4}, \qquad p \in S^n.
\end{equation}
and  represent 
a sequence of round spheres shrinking to a point at time $t=T$.
They are shrinking solitons and  type I ancient solutions.

\smallskip

{\em  King  solutions:} They were  discovered
 by J.R. King \cite{K1}.  They can be expressed on $\R^n$ in closed from (after stereographic projection)   namely 
 $g= \hat  u_{_{K}} (\cdot,t)^{\frac 4{n-2}} \, g_{_{\R^n}}$, where $\hat u_{_{K}}$ is the radial function
 \begin{equation}
\label{eq-king}
\hat u_{_{K}}(y,t) = \left(\frac{a(t) }{1 + 2b(t) \, |y|^2 + |y|^4}\right)^{\frac{n-2}{4}}, \qquad y \in \R^n 
\end{equation}
and the coefficients  $a(t)$ and $b(t)$ satisfy  a certain system of ODEs.
The King solutions are {\em not
solitons} and  may be  visualized, as $t \to -\infty$,    as two Barenblatt self-similar solutions  ''glued'' together
to form a compact solution to the Yamabe  flow. They are type I ancient
solutions.

\medskip
Let us make the analogy  with the Ricci flow on $S^2$. The two  explicit compact ancient solutions to the two dimensional Ricci flow are the contracting spheres   and the King-Rosenau solutions
\cite{K1}, \cite{K2}, \cite{R}. The latter ones are the analogues of the King solution (\ref{eq-king}) to the Yamabe flow. The difference is that the King-Rosenau solutions are type II ancient solutions to the Ricci flow  while
the King solution above is  of type I.

It has been showed by Daskalopoulos, Hamilton and Sesum  \cite{DS1} that the spheres and the King-Rosenau solutions are the only compact ancient solutions to the two dimensional Ricci flow. The natural question to raise is whether the analogous statement holds  true for the Yamabe flow, that is, whether the contracting  spheres and the King solution are the only compact ancient solutions to the Yamabe flow. This occurs not to be the case as the following discussion shows.

\medskip
Indeed, in \cite{DDS} the existence of a new class of type II   ancient radially symmetric solutions of the Yamabe flow \eqref{eq-YFS}
on $S^n$  was shown. These  new solutions, as $ t \to -\infty$, may be visualized as two spheres joined by a short neck. 
Their curvature operator changes sign. We will refer to them as {\em towers of moving bubbles}.

\medskip

Since the towers of moving bubbles are shown to be type II ancient solutions, while the contracting spheres and  the King solutions are of  type I, one may still ask whether the latter two  are
the only ancient compact type I solutions of  the Yamabe flow on $S^n$, equation \eqref{eq-YFS}. In this work we will observe that this is not the case,  as will show  the existence
of other ancient compact type I solutions on $S^n$. 

\medskip

It is simpler to construct  these new solutions in cylindrical coordinates, so let us first describe the coordinate change. 
Let  $g=\hat u^{\frac 4{n-2}} (\cdot,t) \, g_{_{\R^n}}$ be  a radially symmetric  solution of \eqref{eq-ufd}.
For any $T >0$    the  cylindrical change of variables is given by 
\be\label{eqn-cv1}
u(x,\tau )   = (T-t)^{-\frac 1{p-1}} r^{\frac 2{p-1}} \, \hat u(y,t)  , \quad x=\ln  |y|,\,  \tau  = - \ln (T-t).
\ee
In this language equation \eqref{eq-ufd}  becomes
\begin{equation}\label{eqn-vv}
(u^p)_\tau =  u_{xx} +  \alpha u^p -  \beta u, \quad \beta = \frac {(n-2)^2}4 ,\quad \alpha = \frac p{p-1} = \frac {n+2}4.
\end{equation}
By  suitable scaling we can make the two constants $\alpha$ and $\beta$ in \eqref{eqn-vv} equal to 1, so that from now on
we will consider  the equation
\begin{equation}\label{eqn-v}
(u^p)_\tau =  u_{xx} +  u^p -   u.
\end{equation}
Indeed, one can see that
\be\label{eqn-cv2}
u (x,\tau)=\left ( \frac{\alpha}{\beta} \right )^{\frac 1{p-1}}\, \tilde u(\frac x{\sqrt{\beta}}, \frac \tau\alpha )
\ee
solves \eqref{eqn-v} iff the $\tilde u$ solves \eqref{eqn-vv}.

\medskip 

It is  well known (we refer the reader to the book by J.L. Vazquez \cite{V}, Section 3.2.2)  that for any given $\lambda \geq  0$  equation \eqref{eqn-v} admits
an one parameter family of traveling wave  solutions of the form $u_\lambda(x,t) = v_\la(x-\lambda \, t)$ with the behavior
\be
\label{eqn-behavior1}
v_\la(x)  = O(e^{x}), \quad \mbox{as}\,\,\  x \to - \infty.
\ee 
It follows that
$v:=v_\la$ satisfies the equation
\begin{equation}\label{eqn-v2}
v_{xx} + \lambda \, (v^p)_x +  v^p -  v=0.
\end{equation}
The solutions $v_\lambda$ define  Yamabe shrinking solitons which correspond to smooth self-similar solutions of 
\eqref{eq-ufd}  when expressed  as  metrics on $\R^n$ (the smoothness follows from condition \eqref{eqn-behavior1}). 
It was shown in \cite{DS2} that they are type I ancient solutions. 

Solutions of \eqref{eqn-v2} with  $\lambda =0$ correspond to  the steady states of equation \eqref{eqn-vv}
and are given in closed form as  the one parameter family,  
\begin{equation}\label{eqn-w1}
v_0(x) = \left ( \frac{k_n \, c \, e^{\gamma  x}}{1 +c^2 \, e^{2\gamma  x}} \right )^{\frac {n-2}2}, \qquad c>0
 \end{equation}
with  $\gamma= \frac 2{n-2}$ and $k_n = \left (\frac {4n}{n-2} \right )^{1/2}.$
They  represent geometrically the standard metric  on the  sphere. 
\smallskip

When $\lambda >0$,  solutions to \eqref{eqn-v2}  with behavior 
 \eqref{eqn-behavior1} define  smooth complete and non-compact Yamabe solitons
(shrinkers)  which all have {\em cylindrical behavior at infinity}, namely
$$v_\lambda(x) = 1 + o(1), \qquad \mbox{as}\,\, x \to + \infty.$$
In \cite{DKS} the asymptotic behavior, up to second order,   of these solutions was shown. 
In particular,  it follows from Theorem 1.1 in   \cite{DKS},   that  for any  $\la \geq 1$  there exists 
a unique solution $v_{\la}$  of  \eqref{eqn-v} which satisfies 
\be\label{eqn-v0}
v_\la(0)= \frac 12
\ee
and has the asymptotic behavior 
\be\label{eqn-vla}
v_{\la}(x) = O(e^{x}), \quad \mbox{as}\,\,\  x \to - \infty \quad \mbox{and} \quad  v_{\la}(x) = 1 - C_\la\, e^{-\gamma_\la x} +
o(e^{-\gamma_\la x}), \quad \mbox{as}\,\,\  x \to +\infty
\ee
for some constants $\gamma_\la >0$ and $C_\la >0$ (depending on $\la$).  For values of $\lambda$ in the range $0 < \lambda < 1$, the behavior
of the solutions $v_\lambda$ was also studied in \cite{DKS} and differs for dimensions $3 \leq  N  \leq 6$ and $N  \geq 6$.

\bremark For the convenience of the reader let us point out that 
the   proof of Theorem 1.1 in  \cite{DKS} is given in chapter 3 where the  solution $v$ in cylindrical coordinates satisfies 
equation  
\begin{equation}\label{eqn-v5}
{\bar \alpha}^{-1} v_{xx} + \beta (p-1)  \, v^{p-1} \, v_x +  v^p -  v=0
\end{equation}
for a parameter  $\beta >0$  and 
$$\bar \alpha := \frac{(n-2)^2}4=\frac 4{(p-1)^2}.$$
If  
\bee\label{eqn-var2}
\bar  v(x) = v(\gamma x), \qquad \gamma=\bar \alpha^{-1/2}=\frac{p-1}2
\eee
then $\bar v$ satisfies \eqref{eqn-v2} 
with 
$\la ={\beta (p-1)}/({p\gamma}) =2\beta/p.$
Hence $\beta =\beta_1:=p/2$ in Theorem 1.1 in \cite{DKS}   corresponds to 
$\lambda =1$ in our case. 
\eremark

\bremark Since equation \eqref{eqn-v2} is translation invariant the solution $v_\la$ generates an one parameter
family $v_{\la,h}$ of solutions of equation \eqref{eqn-v2} given by $v_{\la,h}= v_\la(x+h)$ which satisfy $v_{\la,h}(-h) =1/2$.
\eremark

Linearizing  equation \eqref{eqn-v2} around the constant solution $v=1$ (which corresponds to the cylinder in geometric terms)
we  obtain the equation  
\begin{equation}\label{eqn-v3}
\tilde  v_{xx} + \lambda p\, \tilde  v_x +  (p-1)\, \tilde  v=0. 
\end{equation}
Hence, assuming that $v \approx 1 - C \, e^{-\gamma_\la x}$, as   $x \to +\infty$, it follows that $\gamma_\la$ satisfies the equation
\be
\label{eqn-eqgamma}
\gamma^2 - \la p \gamma + (p-1) = 0
\ee
and its roots   are non-complex (which corresponds to non-oscillating solutions $\hat v$ of \eqref{eqn-v2}) iff 
$$\la \geq  \frac{2 \, \sqrt{p-1}}{p}.$$
All such solutions were studied in \cite{DKS}, however here we will restrict ourselves to the case 
$$\la \geq 1.$$
It has been shown in \cite{DKS} (Theorem 1.1)  that when $\la  \geq 1$,   the solution $v_{\la}$ of \eqref{eqn-v2} is monotone increasing  
and  satisfies \eqref{eqn-vla}
with 
\begin{equation}
\label{eq-gamma-la}
\gamma_\la = \frac{\la p - \sqrt{\la^2 p^2 - 4 (p-1)}}{2} >0,
\end{equation}
which corresponds to the smallest of the roots of \eqref{eqn-eqgamma}.  
When $\la=1$, equation \eqref{eqn-v2} admits the explicit one parameter family of Barenblatt solutions 
$$ v_{1,c} = \left ( \frac 1{1+ c\, e^{-(p-1)\,x} }\right )^{1/(p-1)}, \qquad c >0$$
where we recall that $p-1=4/(n-2)$ and one may choose $c=c_p$ so that 
$$ v_1 = \left ( \frac 1{1+ c \, e^{-(p-1)\,x} }\right )^{1/(p-1)}$$
satisfies the condition $v_1(0)=1/2$. 
It follows, that  in this case
\be\label{eqn-vla1}
v_1 = 1 - C_1 \, e^{-(p-1)\, x} + o(e^{-(p-1)\, x}), \quad \mbox{as}\,\,\  x \to +\infty
\ee
for a constant  $C_1=C_1(p)>0$. Notice that when $\la =1$  the roots  of \eqref{eqn-eqgamma} are given by
$$\gamma = \frac{ p \mp |p-2|}{2}.
$$
and $\gamma_1:=(p-1)$ in \eqref{eqn-vla} (as it follows from \eqref{eqn-vla1}), hence it satisfies  
$$\gamma_1 = \begin{cases} \frac{p -  |p-2|}{2}, \qquad &\mbox{if} \,\, p \leq 2\\
\frac{p +  |p-2|}{2}, \qquad &\mbox{if} \,\, p > 2.
\end{cases} 
$$
In other words, when $p>2$ (corresponding to $n<6$)  the Barenblatt solution ($\la=1$) satisfies \eqref{eqn-vla}
where $\gamma_1$ is now the largest of the roots of \eqref{eqn-eqgamma}. 

\medskip

We will next give the {\em ansatz} of the construction of new type I solutions of \eqref{eqn-v}, which will
be the main focus in this work. Assume that 
\be\label{eqn-ula0}
u_{\la,h}(x,\tau) := v_\la(x - \la \tau  + h)
\ee
is a traveling wave solution of \eqref{eqn-v} for a parameter $\la \geq 1$ and $h \in \R$.  
Since equation \eqref{eqn-v} is invariant under reflection $x \to -x$, it follows that 
\be\label{eqn-ula00}
\hat u_{\la,h}(x,\tau) := u_{\la,h}(-x,t) = v_\la(-x - \la \tau + h)
\ee
is also a solution to \eqref{eqn-v}. It   corresponds to another  traveling wave of \eqref{eqn-v2} 
which travels in the opposite direction than $u_{\lambda,h}$.
It follows from \eqref{eqn-vla} that $u_{\la,h}$ and $\hat u_{\la,h}$ satisfy the asymptotics 
\be\label{eqn-ula1}
u_{\la,h}(x,\tau) = O(e^{x}), \,\,\,    \mbox{as}\,\,    x \to - \infty \qquad \mbox{and} 
\qquad   \hat u_{\la,h}(x,\tau) = O(e^{-x})
\, \, \,  \mbox{as}\, \,  x \to +\infty. 
\ee
In addition, we have
\be\label{eqn-ula2}
u_{\la,h}(x,\tau) = 1 - C_\la\, e^{-\gamma_\la (x-\la \tau+h)} +
o(e^{-\gamma_\la (x-\la \tau+h)}), \quad    \mbox{as}\, \,  x-\la \tau +h \to +\infty. 
\ee
and also 
\be\label{eqn-ula3}
 \hat u_{\la,h}(x,\tau) = 1 - C_\la\, e^{-\gamma_\la  (-x-\la \tau+h)} + o(e^{-\gamma_\la(-x-\la \tau+h)}), \quad  \mbox{as}\,\,  x + \la \tau -h  \to -\infty. 
\ee
with $\gamma_\lambda$ given by \eqref{eq-gamma-la} and $C_\lambda >0$ depends only on $\lambda$.

\medskip
In this work, we will show the existence of four   parameter class of ancient solutions $u_{\la,\la',h,h'}$ of
equation \eqref{eqn-v} with $\la, \la'  >  1$ and  $h, h' \in \R$, 
which as $t \to -\infty$ may be visualized as the two traveling wave solutions,  $u_{\la,h}$ (traveling on the left) and $\hat u_{\la',h'}$ 
(traveling on the right). In fact, we will show in the next section that  $u_\bla$ is given by
\be\label{eqn-wbla0}
u_\bla = v_\bla - w_\bla
\ee
with $$v_\bla:= \min \big ( u_{\la,h}(\cdot,\tau), \hat u_{\la',h'}(\cdot,\tau) \big )$$
and $w_\bla >0$ an error term which is small in an appropriate norm. 

Let  $g_\bla:=  u_\bla^\mmm\, g_{cyl}$ denote the  metric on the 
cylinder $\R  \times S^{n-1}$ defined by the solution $u_\la$ of \eqref{eqn-v}. Here $g_{cyl}:= dx^2 + g_{_{S^{n-1}}}$ denotes the standard  
cylindrical metric. We have seen  that \eqref{eqn-v} is equivalent to $g_\bla$ 
satisfying the rescaled Yamabe flow $g_t = -(R - 1)g$. In addition we will show that each  $u_\bla$,  when lifted on $S^n$,  defines a {\em smooth
ancient type I}  solution to the Yamabe flow on $S^n \times (-\infty,T)$, in the sense that the norms of its curvature operators are uniformly bounded in time (which exactly means the corresponding solution to the unrescaled Yamabe flow \eqref{eq-YF} is a type I ancient solution in the sense of Definition \ref{defn-ancient}).
 Our main result is summarized as follows.

\smallskip

\begin{thm}\label{thm} For any $(\bla) \in \R^4$ such that $\la, \la' >1$ there exists an ancient solution $u_\bla$ of \eqref{eqn-v} 
defined on $\R \times (-\infty,T)$,  for some $T=T_\bla \in (-\infty, +\infty]$  and satisfies 
$$0 < u_\bla \leq v_\bla, \qquad \mbox{for all} \,\, (x,\tau) \in \R \times (-\infty, T).$$
In addition, the metric $g_\bla:= u^\mmm\, g_{cyl}$ when lifted on $S^n$ defines a smooth ancient  solution of the rescaled Yamabe flow $g_t = -(R-1)\, g$,
on $S^n \times  (-\infty, T)$. This is a  type I ancient solution in the sense that the norms of its curvature operators are uniformly bounded in time (which exactly means the corresponding solution to the unrescaled Yamabe flow \eqref{eq-YF} is a type I ancient solution in the sense of Definition \ref{defn-ancient}).
\end{thm}  

\smallskip 

The {\em organization}  of the paper is as follows : in section \ref{sec-anzatz}  we prove Theorem \ref{thm1} which is  the existence of a  four parameter family of ancient solutions $u_\bla$. 
 In particular, we show that  each of them is exponentially close in the integral sense to  a given approximating solution which depends on the four parameters $\bla$. 
  In section \ref{sec-geom} we show that all our constructed solutions are Type I ancient solutions, as stated in Theorem \ref{thm2}.  Theorem \ref{thm} is a direct consequance of Theorems \ref{thm1} and \ref{thm2}. 

\smallskip

\bremark[KPP equation and the work of Hamel-Nadirashvili \cite{HN}] 
Equation \eqref{eqn-v}  resembles  the well known  semilinear KPP equation
\be\label{eqn-kpp} 
u_t = u_{xx} + f(u)
\ee
for a nonlinearity $f(u)$ which satisfies appropriate growth conditions (c.f. in \cite{HN}). It is well known that equation \eqref{eqn-kpp}
possesses a family of traveling wave solutions $v_\la$, $\la \geq \la_*$ with similar behavior as those of equation\eqref{eqn-v} described above. 
F. Hamel and N. Nadirashvili, in \cite{HN}, constructed ancient solutions $u_\bla$ of  equation\eqref{eqn-kpp}. The main idea in \cite{HN} is to exploit the semilinear character of equation \eqref{eqn-kpp} and estimate the error
or approximation $w_\bla$ as in \eqref{eqn-wbla0} by the solution to  the linear equation 
$$\nu_t = \nu_{xx} + f'(0)\, \nu.$$
This   allows them to estimate the error of approximation $w_\bla$ pointwise in a rather precise manner. 
However, the same method cannot be applied to   our quasilinear equation \eqref{eqn-v}, which actually becomes {\em singular}  as
$x \to \pm \infty$ (where the approximating  supersolution $v_\bla$ vanishes). In this work we need to depart from the methods in \cite{HN} and 
we have chosen to use integral methods in order to estimate  the error term $w_\bla$.

\eremark

\section{The construction of merging traveling waves}
\label{sec-anzatz}

For fixed   $\la, \la' \geq 1$, $h, h' \in \R$, let $u_{\la,h}$ and  $\hat u_{\la',h'}$ 
be the two traveling wave solutions of equation \eqref{eqn-v} as  introduced in the previous section.  
We define the {\em approximating supersolution}  $v:=v_{\la, \la', h, h'}$ by 
\be\label{eqn-va1}
v_{\la, \la', h, h'}(\cdot, \tau)  = \min \big ( u_{\la,h}(\cdot,\tau), \hat u_{\la',h'}(\cdot,\tau) \big ), \qquad  \tau \in (-\infty, +\infty).
\ee
Using  the definitions of $u_{\la,h}$ and $\hat u_{\la',h'}$ we have
\be\label{eqn-va2}
v_{\la, \la', h, h'}(\cdot, \tau) =  \min \big ( v_\la(x- \la \tau  +h), v_{\la'}(-x- \la' \tau+h')  \big ).
\ee

We will show in this section that for any $(\bla) \in \R^4$ such that $\la, \la' >1$, there exists a solution $u_\bla$ which is close in certain
sense to the approximating supersolution $v_\bla$, as stated next.

\begin{thm}\label{thm1} For any $(\bla) \in \R^4$ such that $\la, \la' >1$ there exists an ancient solution $u_\bla$ of \eqref{eqn-v} 
defined on $\R \times (-\infty,T_\bla)$ for some $T_\bla \in (-\infty, +\infty)$  which satisfies 
$$0 < u_\bla \leq v_\bla, \qquad \mbox{for all} \,\, (x,\tau) \in \R \times (-\infty, T_\bla).$$
In addition, for $\tau <<0$, the solution $u_\bla$ is close to the approximating supersolution $v_\bla$ in the sense that 
$$\int_\R |v^p_\bla - u^p_\bla| (\cdot, \tau)  \, dx \leq D_\bla \, e^{d\tau}$$
where $d=\frac{\gal \gall + (p-1)}{p}$ and $D_\bla$ is a  positive constant depending only on the dimension $n$ and $\bla$. 
Moreover, if $\bla \neq  \blaa$, then $u_\bla \neq u_\blaa$. 
\end{thm}  

We have seen in the introduction that each $u_{\la,h}$ and $\hat u_{\la',h'}$ satisfy  conditions \eqref{eqn-ula1}-\eqref{eqn-ula3}. 
It follows that for each $\tau$  there is a unique  {\em intersection point }
$x(\tau)$ for which $u_{\la,h} (x(\tau),\tau) =\hat u_{\la',h'} (x(\tau), \tau)$.

\begin{lem}\label{lem-xtau} The intersection point $x(\tau)$ of $u_{\la,h}$ and $u_{\la',h'}$ satisfies, as $\tau \to -\infty$,  the asymptotic behavior
\begin{equation}
\label{eq-intersection}
x(\tau) =  \frac{ \gamma_{\lambda} - \gamma_{\lambda'}}{p} \, \tau 
+ \frac{1}{\gamma_{\lambda} + \gamma_{\lambda'}}\, \left(  \ln\frac{C_{\lambda}}{C_{\lambda'}} 
+ h'\gamma_{\lambda'} - h\gamma_{\lambda} \right )  + o(1).
\end{equation}
In addition at $x=x(\tau)$ we have 
\be\label{eqn-u12} 
u_{\la, h}(x(\tau), \tau) = \hat u_{\la', h'}(x(\tau), \tau) = 1- C_\bla  \, e^{d\, \tau} + o(e^{d\tau})
\ee
with
\be\label{eq-d}
d:= \frac {\gamma_{\lambda} \gamma_{\lambda'} + (p-1)}p.
\ee
and $C_\bla$ depending on $\bla$. It also follows that
\be\label{eqn-uder1} 
(u_{\la, h})_x(x(\tau), \tau) = \gamma_\la \, C_\bla  \, e^{d\, \tau} + o(e^{d\tau}), \qquad (\hat u_{\la', h'})_x(x(\tau), \tau) = -\gamma_{\la'}
 \, C_\bla  \, e^{d\, \tau} + o(e^{d\tau}). 
\ee
\end{lem}

\begin{proof} Using  the asymptotic behavior \eqref{eqn-ula2} and \eqref{eqn-ula3} it follows  that at $x=x(\tau)$ we have 
\[  C_{\lambda} e^{-\gamma_{\lambda}\, (x - \lambda\tau +h)} + o(e^{-\gamma_{\lambda}\, ( x + h - \lambda \tau)})
\approx  C_{\lambda'} e^{-\gamma_{\lambda'}\, (-x - \lambda'\tau -h')} + o(e^{-\gamma_{\lambda'}\, (-x' - \lambda'\tau -h')}).\]
Solving for $x$ readily implies that
$$x(\tau) =  \frac{\lambda \gamma_{\lambda} - \lambda' \gamma_{\lambda'}}{\gamma_{\lambda} + \gamma_{\lambda'}} \, \tau 
+ \frac{1}{\gamma_{\lambda} + \gamma_{\lambda'}}\, \left(  \ln\frac{C_{\lambda}}{C_{\lambda'}} 
+ h'\gamma_{\lambda'} - h\gamma_{\lambda} \right )  + o(1).$$
Using  equation \eqref{eqn-eqgamma}, we may eliminate  the $\lambda, \lambda'$  from the above expression substituting
$$\la \gamma_\lambda = \frac{\gamma^2_\lambda  + (p-1)}{p}, \qquad \la'  \gamma_{\lambda'} = \frac{\gamma_{\lambda'}^2 + (p-1)}{p}$$
and obtain  \eqref{eq-intersection}. With this choice of $x(\tau)$ we have 
$$ \gamma_\lambda \big ( x(\tau) - \lambda \, \tau \big ) = \gamma_\lambda 
 \left ( \frac{ \gamma_{\lambda} - \gamma_{\lambda'}}{p} - \lambda \right ) \, \tau +  c_\bla$$
for some constant $c_\bla$ depending on $\bla$ and eliminating  $\lambda$ as above we obtain
$$ \gamma_\lambda \big ( x(\tau) - \lambda \, \tau \big ) = - \frac{\gamma_\la \gamma_{\la'} + (p-1)}{p}\, \tau + c_\bla.$$
Setting $d:=\frac{\gamma_\la \gamma_{\la'} + (p-1)}{p}$, we conclude using \eqref{eqn-ula2} that 
$$u_{\la,h} (x(\tau), \tau)  =  1- C_\bla \, e^{d\tau} + o(e^{d\tau})$$
for a constant $C_\bla >0$ depending on $\bla$. 
Since $u_{\la,h} (x(\tau), \tau) = u_{\la',h'} (x(\tau), \tau)$, the \eqref{eqn-u12} follows. 

It remains to show \eqref{eqn-uder1}. Recall that $u_{\la,h}(x,\tau) = v_\la(x-\la\tau+h)$. First we claim that 
\begin{equation}
\label{eq-der-zero}
\lim_{x\to+\infty} (v_{\la})_x = 0.
\end{equation}
To prove the claim note that by \eqref{eqn-v2} we have
\[\left( (v_{\la})_x + \la v_{\la}^p\right)_x = v_{\la} - v_{\la}^p \ge 0,\]
since $v_\la \le 1$, implying there exists a finite limit $\lim_{x\to +\infty} ((v_\la)_x + \la v_\la^p)$ and hence the $\lim_{x\to +\infty} (v_\la)_x = c$. We claim $c = 0$. Indeed, if $c > 0$, there would exist an $x_0$ so that for all $x \ge x_0$ we would have $(v_{\la})_x \ge  c/2$. This would imply that 
\[v_{\la}(x) = v_{\la}(x_0) + \int_{x_0}^x (v_{\la})_x\, dx \ge \frac c2 (x - x_0), \qquad x\ge x_0\]
contradicting that the $\lim_{x\to +\infty} v_{\la}(x) = 1$. Using that $v_\la > 0$ we argue similarly in the case we assume $c < 0$.

We will next prove more precise asymptotics on the derivatives of $v_{\la}$, which will yield \eqref{eqn-uder1}. By \eqref{eqn-v2} we have
\[\left( (v_{\la})_x + \la v_{\la}^p\right)_x = v_{\la} - v_{\la}^p.\]
On the other hand, by \eqref{eqn-vla} we have
\[v_{\la} - v_{\la}^p = C_{\la}\, (p - 1)\, e^{-\gamma_{\la}x} + o(e^{-\gamma_{\la}x}), \qquad \mbox{for} \,\,\,\, x >> 1\]
and hence,
\[\left( (v_{\la})_x + \lambda \, v_{\la}^p\right)_x = C_{\la}\, (p - 1)\, e^{-\gamma_{\la}x} + o(e^{-\gamma_{\la}x}).\]
Integrating it from $x$ to $+\infty$ and using \eqref{eq-der-zero} and that the $\lim_{x\to +\infty} v_{\la}(x) = 1$ yields
\[(v_\la)_x = \la -  \la \, v_{\la}^p - \frac{C_{\la}\, (p-1)}{\gamma_{\la}}\, e^{-\gamma_{\la} x} + o(e^{-\gamma_{\la}x}).\]
Asymptotics \eqref{eqn-vla} implies $v_{\la}^p = 1 - p \, C_{\la} e^{-\gamma_{\la}x} + o(e^{-\gamma_{\la}x})$, and therefore,
\begin{equation}
\label{eq-vlax}
\begin{split}
(v_{\la})_x &= C_{\la}\, e^{-\gamma_{\la}x}\, \left(p\la - \frac{p-1}{\gamma_{\la}}\right) + o(e^{-\gamma_{\la} x}) \\
&= C_{\la}\gamma_{\la} e^{-\gamma_{\la}x} + o(e^{-\gamma_{\la} x}), \qquad \mbox{as} \,\,\,\, x\to +\infty,
\end{split}
\end{equation}
where we have used that $p\la \gamma_{\la} = \gamma_{\la}^2 + (p - 1)$.  Finally, since $x(\tau) - \la\tau +h >> 1$ for $\tau << -1$, we get \eqref{eqn-uder1} by plugging $x(\tau) - \la\tau +h$ in \eqref{eq-vlax}. 
\end{proof}

Denote briefly by $v := v_\bla$. Then we have the following integral identity.

\begin{lem} 
\label{lem-lptau}
We have 
\begin{equation} 
\label{eqn-lptau1} 
\frac {d}{d\tau} \int_{\R} v^p \, dx =  \int_{\R} v^p  \, dx -  \int_{\R} v \, dx  + (\gamma_\la + \gamma_{\la'})  \, C_\bla  \, e^{d\, \tau} + o(e^{d\tau}). 
\end{equation}
\end{lem}

\begin{proof} For simplicity set $u_1:= u_{\la,h}$ and $u_2:=u_{\la',h'}.$ Then   $u_1(\cdot, \tau), u_2(\cdot, \tau)$  are solutions to \eqref{eqn-v}
on $(-\infty, x(\tau))$, $(x(\tau), + \infty)$ respectively and by definition we have $v=u_1$ on $(-\infty, x(\tau))$ and 
$v=u_2$ on $(x(\tau),+\infty). $  In addition, because of \eqref{eqn-ula1} we have
$$\lim_{x \to -\infty} (u_1)_x (x, \tau) = \lim_{x \to +\infty} (u_2)_x (x, \tau) =0.$$
Note this can be proved in the same way as we have proved \eqref{eq-vlax}, just using the asymptotics of our solitons at $x \to -\infty$ instead of $x\to +\infty$.
Hence, integrating  equation  \eqref{eqn-v} for $u_1$ on $(-\infty, x(\tau))$ and  equation  \eqref{eqn-v} for $u_2$ on $(x(\tau), + \infty)$
we obtain
$$\frac {d}{d\tau} \int_{-\infty}^{x(\tau)} u_1^p \, dx =  \int_{-\infty}^{x(\tau)} u_1^p  \, dx -  \int_{-\infty}^{x(\tau)} u_1 \, dx +
(u_1)_x (x(\tau),\tau) + x'(\tau) \, u_1^p (x(\tau),\tau)$$
and
$$\frac {d}{d\tau} \int^{+\infty}_{x(\tau)} u_2^p \, dx =  \int^{+\infty}_{x(\tau)}  u_2^p  \, dx - \int^{+\infty}_{x(\tau)}  u_2 \, dx - 
(u_2)_x (x(\tau),\tau) -  x'(\tau) \, u_2^p (x(\tau),\tau).$$
Since $u_1(x(\tau), \tau)=u_2(x(\tau),\tau)$, adding  the last  two equalities yields
\bee 
\frac {d}{d\tau} \int_{\R} v^p \, dx =  \int_{\R} v^p  \, dx -  \int_{\R} v \, dx +
(u_1)_x (x(\tau),\tau) - (u_2)_x (x(\tau),\tau). 
\eee
Combining this with \eqref{eqn-uder1} readily yields   \eqref{eqn-lptau1}.  
\end{proof}

For any $m \in \mathbb{N}$, let  $u_m$ denote the solution of the initial value problem 
\be\label{eqn-um}
\begin{cases}
(u^p)_\tau =  u_{xx} - u + u^p \qquad & x\in  \R, \,\,  \tau > -m \\
\, u(\cdot,-m) = v_\bla(\cdot, -m) \qquad & x \in  \R.
\end{cases}
\ee
with exponent
$p= \frac{n+2}{n-2} >1.$

\begin{lem}[Uniform barrier from above]\label{lem-1} The solution $u_m$ exists for all time $-m \leq \tau < +\infty$
and satisfies
\be\label{eqn-baum}
u_m \leq  v_\bla
\ee 
\end{lem}
\begin{proof}
The bound \eqref{eqn-baum}  simply follows from the comparison principle.  Since  $u_m(\cdot,-m) \leq u_{\la,h}(\cdot,-m)$ and 
$u_m (\cdot, -m) \leq \hat u_{\la',h'}(\cdot, -m)$
we have
$u_m \leq u_{\la,h}$  and  $u_m \leq \hat u_{\la',h'}$ for $\tau \geq  -m$, concluding that
$u_m \leq v_{\la,\la',h,h'}(\cdot,\tau)   := \min \big ( u_{\la,h}(\cdot,\tau), \hat u_{\la',h'}(\cdot,\tau) \big ).$
The bound \eqref{eqn-baum} and standard arguments on quasilinear parabolic pde   imply that the solution $u_m$ exists for all  $-m \leq \tau < +\infty$. 
\end{proof} 

\medskip
\begin{remark}
{\em In what follows we show a bound  from below for $u_m$,  which is uniform in $m$ and will guarantee that the solutions
$u_m$ will stay positive for $-m \leq \tau < T$, for some uniform in $m$ time $T$. }
\eremark

\bigskip

\begin{lem}[The profile for  $u_m$]\label{lem-3}  
There exists a number $T_m \leq +\infty$ such that $u_m >0$ on $\R \times [-m,T_m)$ and in the 
cases were $T_m < +\infty$, $u_m \equiv 0$ for $ \tau \geq  T_m$. In addition, for all $\tau \in [-m,T_m)$, 
$u_m(\cdot,t)$ satisfies the asymptotic behavior
\be\label{eqn-um3}
u_m(x,\tau)= O(e^{x}), \,\, \,  \mbox{as}\,\,    x \to - \infty  \qquad \mbox{and} \qquad u_m(x,\tau)= O(e^{-x}), \,\, \,  \mbox{as}\,\,    x \to + \infty.
\ee
Moreover,  the function $u_m(x,\tau)$ is decreasing in $\tau$,  for all $\tau > -m$.

\end{lem}

\begin{proof} The first two claims  in this   lemma readily follow from well known results for fast-diffusion equations and the Yamabe flow on $S^n$,
since  $g = u_m(\cdot,\tau)^{\frac{4}{n-2}}\, g_{cyl}$ corresponds to a solution of the Yamabe flow and the behavior  \eqref{eqn-um3} is equivalent to 
saying that  $g$ can be lifted to a smooth metric on  $S^n$. 

We will next show the monotonicity in $\tau$ of the solutions $u_m$. 
It follows from \eqref{eqn-v} that the  function $w_1(x)=v_\la(x+ \la m+h)$ satisfies
\bee\begin{split}
w_1'' + w_1^p - w_1&= v_\la''(x+ \la m+h) + v_\la^p(x+ \la m+h) - v_\la( x+ \la m+h) \\
&=   - \la \, v_\la'(x+ \la m +h) <0
\end{split}
\eee
and similarly  the function $w_2(x)=v_{\la'}(- x - \la' m+h')$ satisfies
\bee\begin{split}
w_2'' + w_2^p - w_2&= v_{\la'}''(-x- \la' m+h) + v_{\la'}^p(-x - \la' m+h) - v_{\la'}( -x- \la' m+h) \\
&=   - \la' \, v'_{\la'} (-x - \la' m +h) <0
\end{split}
\eee
since $v_\la' >0$ for all $\la  \geq 1$. It follows that $f_m:= \min (w_1,w_2)$ is a supersolution, namely it satisfies
$$f_m '' + f_m^p - f_m <   0$$
in the distributional sense. This implies the function $u_m(x,\tau)$ is decreasing in $\tau$ for any $\tau > -m$, $x\in \mathbb{R}$.
Hence, the result follows by a simple approximation argument.  

\end{proof}

\bremark  Each  solution $u_{\la,h}$  satisfies $(u_{\la,h})_\tau \leq 0$, since $(u_{\la,h})_\tau = - \la \, v_\la'(x-\la t+h) < 0$,
because $v_\la' < 0$. 

\eremark

\bremark The inequality $(u_m)_\tau \leq 0$ implies that the scalar curvature $R_m$ of the
corresponding metric defined by the solution $u_m$ is nonnegative. 
Recall that  for a solution $u$ of \eqref{eqn-v}, $R \geq 0$ corresponds to $(u^p)_\tau \leq u^p$. 
\eremark

%

We will next show that each $u_m$ is sufficiently close to $v_\bla$ in certain sense and this happens uniformly in $m$. This will assure that
the limit as $m \to +\infty$  is a non-trivial solution of \eqref{eqn-v}. 
We begin with the  following crucial for our purposes estimate which is a consequence of Lemma \ref{lem-lptau}. 

\begin{prop} We have 
\be\label{eqn-qtau5}
 Q_m(\tau):=  \int_{\R} ( v^p - u_m^p) \leq D_\bla \, e^{d\tau}
 \ee
for a constant $D_\bla >0$ depending only on $\bla$.
\end{prop} 

\begin{proof} Since $u_m$ satisfies \eqref{eqn-v}, integrating this equation on $\R$ readily yields
$$\frac {d}{d\tau} \int_{\R} u^p \, dx =  \int_{\R} u^p  \, dx -  \int_{\R} u \, dx.$$
Here we used that
$$\lim_{x \pm \infty} (u_m)_x (x, \tau)=0$$
which easily follows from \eqref{eqn-um3} and the fact that the metric $u_m^{\frac{4}{n-2}}\, g_{cyl}$ when lifted to a sphere defines a smooth metric.
If we combine this with \eqref{eqn-lptau1} we obtain
\be\label{eqn-umtau}
\frac {d}{d\tau} \int_{\R} (v^p - u_m^p)  \, dx  \leq   \int_{\R} (v^p - u_m^p)  \, dx -  \int_{\R} (v-u_m)  \, dx  +
\bar C_\bla \, e^{d\, \tau}
\ee
with $\bar C_\bla :=  (\gamma_\la + \gamma_{\la'})  \, C_\bla +1 $. 
Next set $w_m:=v-u_m$ and observe that since $u_m \leq v$  we have $w_m \geq 0$. 
Since 
\[(v^p  - u_m^p)\,  = a\,  (v_\delta  - u_m), \qquad \mbox{where} \qquad a := p\, \, \int_0^1 (s\, v   +(1-s)\, u_m)^{p-1}\, ds\] 
we may write \eqref{eqn-umtau} as 
\[\frac{d}{d\tau}\int_{\mathbb{R}} a \, w_m \, dx =  \frac{p-1}{p}\int_{\mathbb{R}} a \, w_m \, dx +  \frac1p\, \int (a - p) w_m\, dx +  
\bar C_\bla \, e^{d\, \tau} + o(e^{d\tau}).\] 
Note that,  since both $v  \leq 1$ and $u_m \leq 1$, we have 
\[
a - p = p\, \left( \int_0^1 (s\, v  + (1-s)\, u_m)^{p-1}\, ds - 1\right) \leq 0. 
\]
Hence, using also that $w_m \geq 0$, we conclude 
$$
\frac{d}{d\tau}\int_{\mathbb{R}} a\,  w_m \, dx \le \frac{p-1}{p}\int_{\mathbb{R}} a \, w_m\, dx  + \bar C_\bla \, e^{d\, \tau}.
$$
Setting 
$$Q_m(\tau):= \int_{\R} (v^p - u_m^p)(\cdot, \tau)   \, dx = \int_{\R} a\, (v - u_m) (\cdot,\tau)\, dx$$
we obtain
\bee
\frac {d}{d\tau} Q_m(\tau) \leq \frac{p-1}{p} \, Q_m(\tau) +  \bar C_\bla \, e^{d\, \tau}.
\eee
Equivalently, if
$$\hat Q_m(\tau):= e^{-\frac{(p-1)\tau}{p}} Q_m(\tau)$$
and $\mu:= d  - \frac{p-1}{p}$ we have
$$
\frac {d}{d\tau} \hat Q_m(\tau) \leq  \bar C_\bla \, e^{ \mu \, \tau}  + o(e^{\mu \tau}).
$$
Next observe  that by \eqref{eq-d} we have   $d > \frac{p-1}{p}$, hence $\mu >0$. Also,  since $w_m=0$ at $\tau=-m$, we have $\hat Q_m(-m)=0$.
Hence, the above differential inequality yields the bound 
$$\hat Q_m(\tau) \leq \mu^{-1} \bar C_\bla  \, e^{ \mu \, \tau}, \qquad  \tau > -m$$
from which the bound \eqref{eqn-qtau5} readily follows. 
\end{proof}

\begin{prop}[Passing to the limit]\label{prop-limit} After passing to a subsequence, the sequence $\{ u_m \}$ 
converges, uniformly on compact subsets of $\R \times (-\infty, +\infty)$, to an ancient  solution $u=u_\bla$ of \eqref{eqn-v}.
It is positive, $u >0$, on $\R \times (-\infty, T_\bla)$ for some $T_\bla$, depending only on $\bla$ and the dimension $n$. 
In addition, $u(x,\tau)$ is decreasing in $\tau$ for all $(x,\tau) \in \R \times (-\infty, +\infty)$ and satisfies conditions   \eqref{eqn-um3}.
\end{prop}
\begin{proof} The uniform bound $u_m \leq v$ implies that the sequence of solutions $\{ u_m \}$ is uniformly bounded on compact subsets of 
$\R \times (-\infty, + \infty)$, hence by standard estimates it is equicontinuous. Hence, passing to a subsequence  it converges to a limit
$u=u_\bla$ and $u(x,\tau)$ is decreasing in $\tau$ for all $x \in \R$, since the same holds for each $u_m$ be the previous lemma. 

We will next show that the limit $u$ is non trivial. Since $u_m (\cdot, \tau) \leq v(\cdot, \tau)$, 
$v^{p-1}  \leq  a \leq  p\, v^{p-1}$
and $v(\cdot, \tau) \leq C(\tau) \, \min \, ( e^{x}, e^{-x} )$, we can pass to the limit $m \to +\infty$ in \eqref{eqn-qtau5} and
using the dominated 
convergence theorem we obtain the bound 
\be\label{eqn-qtau6}
Q (\tau) :=   \int_{\R} \hat a\, (v - u) (\cdot,\tau)\, dx  \leq  D_{\lambda,\lambda',h,h'}\, e^{d\tau}, 
\ee
for $$\hat a =   p\,  \int_0^1 (s\, v + (1-s)\, u)^{p-1}\, ds = \frac{v^p - u^p}{v-u}.$$
Observe that
$$Q(\tau):=  \int_{\R} a \, (v-u) (\cdot,\tau)  \, dx =  \int_{\R} v^p (\cdot,\tau) \, dx -  \int_{\R} u^p(\cdot,\tau) \, dx.$$
In particular, this implies for every fixed $\tau << -1$  there exists a point $(x,\tau)$, such that 
$x  \in [x(\tau)-1,x(\tau)]$ and  
$$0 \leq v(x ,\tau) - u(x,\tau) \leq  D_{\lambda,\lambda',h,h'} \, e^{d\tau}.$$
Recalling that  $v(x,\tau) \approx  1 -  C_\lambda e^{\lambda \gamma_\lambda \tau} 
e^{-\gamma_\lambda (x+h)}$, whenever $x - \lambda\tau >> 1$ and $x\in [x(\tau) - 1, x(\tau)]$, we conclude that 
$$u(x,\tau) \geq 1 - \bar{D}_{\lambda,\lambda',h,h'} \, e^{d\tau}.$$
This implies that 
\begin{equation}
\label{eq-m-below}
m(\tau):= \max_{\R} u(\cdot,\tau)  \ge 1 - \bar{D}_{\lambda,\lambda',h,h'}\, e^{d\tau}.
\end{equation}
On the other hand,  using \eqref{eqn-u12} we have
$$u(x,\tau) \leq v(x,\tau) \leq  v(x(\tau),\tau) = 1 - C_\bla \, e^{d\tau} + o(e^{d\tau}).$$
Hence,
\begin{equation}
\label{eq-2-sided}
1 - \bar{D}_{\lambda,\lambda',h,h'}\, e^{d\tau} \leq  m(\tau)  \leq 1 - \bar C_\bla \, e^{d\tau}.
\end{equation}
for $\bar C_\bla:= C_\bla +1$. This in particular implies that $m(\tau) >0$ for all $\tau \leq  \tau_0$ if  $\tau_0 <<0$. Hence, there exists
a number $T=T_\bla$ such that $m(\tau) >0$ for all $t \leq T_\bla$ and we may assume that $T_\bla$ is the maximal such time (note that 
$T_\bla$ may be equal to $+\infty$). Standard estimates then imply that $u(x,\tau) >0$ for all $(x,\tau) \in \R \times (-\infty, T_\bla)$. We also have that $u(\cdot, \tau)$ satisfies conditions \eqref{eqn-um3}.
\end{proof}

Next we show how to distinguish between solutions that we have constructed using different parameters. More precisely, we have the following result.

\begin{prop}[Distinguishing between solutions] \label{prop-distinguish} Let   $\lambda,\lambda', \bar \lambda, \bar \lambda' >1$ and  $(\lambda, \lambda',h,h') 
\neq (\bar \lambda, \bar  \lambda', \bar h, \bar h')$, then $u_{\bla}  \neq  u_{\blaa}$.  
\end{prop}

\begin{proof}
We will prove the Proposition in two steps.

\begin{step}
\label{step-la}
Fix $h, h', \bar h, \bar h'$. If $\lambda,\lambda', \bar \lambda, \bar \lambda' >1$ and 
$(\lambda,\lambda') \neq ( \bar \lambda, \bar \lambda')$,  then $u_{\bla} \neq u_{\blaa}$. 
\end{step}
To prove the claim we argue by contradiction. Assume that  $(\lambda,\lambda') \neq ( \bar \lambda, \bar \lambda')$ and  $u_{\bla} \equiv  u_{\blaa}$. 
For simplicity we call this solution $u$.  Without loss of generality we may assume that $\lambda < \bar \lambda$. By \eqref{eq-d} we have 
$$d = \frac{\gal \gall + (p-1)}p =  \frac{\gabl \gabll + (p-1)}p = \bar d$$
implying $\gal \gall = \gabl \gabll$. If  $m(\tau):= \max_{\R} u  (\cdot,\tau)$ then it satisfies \eqref{eq-2-sided}.
Let $x_{\max}  (\tau) $ be a point such that $m(\tau)= u (x_{\max}(\tau),\tau)$. 
It  $v, \bar v$  are the approximating solutions corresponding to $u_\bla, u_{\blaa}$ respectively, then we have 
$u=u_{\bla} \leq v$ and $u=u_{\blaa} \leq \bar  v$, which  combined with \eqref{eq-2-sided} gives that
\be\label{eqn-v25}
v(x_{\max}(\tau),\tau) \geq 1-C_1\, e^{d\tau} \quad \mbox{and} \quad \bar  v(x_{\max}(\tau),\tau) \geq 1-C_1\, e^{d\tau}.
\ee
In addition, if $x(\tau)$, $\bar x(\tau)$ denote the maximum points of $v(\cdot,\tau)$, $\bar v(\cdot,\tau)$ respectively, we have 
\begin{equation}
\label{eqn-v26}
v(x_{\max} (\tau),\tau) \leq v(x(\tau),\tau) \leq  1- C_2\, e^{d\tau}  \quad \mbox {and} \quad 
\bar v(x_{\max} (\tau),\tau) \leq  \bar v(\bar x(\tau),\tau)  \leq 1- C_2\, e^{d\tau}.
\end{equation}
Using the asymptotics \eqref{eqn-vla} and the estimates \eqref{eqn-v25}, \eqref{eqn-v26}  we conclude that 
$$ e^{-\gamma_\lambda (x_{\max}(\tau) -\la \tau )} \approx  e^{-\gabl  (-x_{\max}(\tau)  -\bar \la \tau )},$$
yielding
\be\label{eqn-xtau11}
x_{\max}(\tau) = \frac{\gamma_\lambda - \gamma_{\lambda'} + (p - 1)}{p}\, \tau + O(1) = 
\frac{\gamma_{\bar{\lambda}} - \gamma_{\bar{\lambda'}} + (p - 1)}{p}\, \tau + O(1).\
\ee
This in particular implies that $\gamma_\lambda - \gamma_{\lambda'} = \gamma_{\bar{\lambda}} - \gamma_{\bar{\lambda'}}$. 
In addition, by \eqref{eqn-xtau11} and   \eqref{eq-intersection}  we have  
$$x_{\max}(\tau)  \approx x(\tau) + O(1) = \bar{x}(\tau) + O(1).$$
On the other hand,   by  \eqref{eqn-qtau5} we have 
\[\int_{\mathbb{R}} (v^p - u^p) \, dx \le C_1 e^{d\tau}\]
and
\[\int_{\mathbb{R}} (\bar{v}^p - u^p)\, dx \le C_2 e^{d\tau}.\]
Recalling that $u \leq v$ and $u \leq \bar v$ we conclude that 
\[\int_{\mathbb{R}} |v^p - \bar{v}^p|\, dx \le C \,  e^{d\tau}.\]
Since $x(\tau)$ and $\bar{x}(\tau)$ are comparable for $\tau << -1$, without a loss of any generality we may assume $x(\tau) \le \bar{x}(\tau)$. 
Then we have
\[\int_{-\infty}^{x(\tau)} |v_{\lambda}^p(x + h - \lambda\tau) - v_{\bar{\lambda}}^p(x + \bar{h} - \bar{\lambda}\tau)|\, dx \le C e^{d\tau}\]
implying
\[\int_{-\infty}^{x(\tau) + h - \lambda\tau} |v_{\lambda}^p(y) - v_{\bar{\lambda}}^p(y + (\lambda - \bar{\lambda})\tau + \bar{h} - h)|\, dy \le C\, e^{d\tau}.\]
Using the asymptotics \eqref{eqn-vla} and that $(\lambda - \bar{\lambda}) \, \tau >> 1$ (since $\la < \bar \la$), the previous inequality gives 
\begin{equation}
\label{eq-int-la}
\int_M ^{2M} |e^{-\gamma_\lambda y} - e^{-\gamma_{\bar{\lambda}} (y + \bar{h} - h + (\lambda - \bar{\lambda})\tau)}| \, dy \le C e^{d\tau},
\end{equation}
for a big constant $M >> 1$. Estimate \eqref{eq-int-la} holding for any $\tau << -1$ forces $\lambda = \bar{\lambda}$, which concludes the proof of Step \ref{step-la}.

\begin{step}
\label{step-h}
Fix now $\lambda, \lambda' > 1$. If   $h, h', \bar h, \bar h'$ satisfy $(h,h') \neq (\bar h, \bar h')$,  then $u_{\bla} \neq u_{\blaaa}$. 
\end{step}

To prove Step \ref{step-h} we argue by contradiction. Assume  that $(h,h') \neq (\bar h, \bar h')$ and $u:= u_{\bla} = u_{\blaaa}$. 
By translating $v_\lambda, v_{\lambda'}$  by $\bar h, \bar h'$ respectively,  we may assume that $\bar h = \bar h'=0$ 
(our proof is not using the exact choice of $v_\lambda, v_{\lambda'}$ so that $v_\lambda(0)= v_{\lambda'}(0)=1/2$). 
Let $v$ be the approximation of $u:=u_\bla$   given by $v := \min (v_{\lambda}(x + h - \lambda\tau), v_{\lambda'}(-x +h' - \lambda'\tau))$. 
We observe that
$$Q(\tau):=  \int_{\R} a \, (v-u) (\cdot,\tau)  \, dx =  \int_{\R} v^p (\cdot,\tau) \, dx -  \int_{\R} u^p(\cdot,\tau) \, dx.$$
Hence, the bounds $u \leq v$ and \eqref{eqn-qtau6} yield
$$\big | \int_{\R} v^p(\cdot,\tau)  \, dx  - \int_{\R} u^p(\cdot,\tau)  \, dx \big | \leq  C_\bla \, e^{d \tau},$$
where $d$ is given by \eqref{eq-d} and depends only on $\lambda, \lambda'$.  
Similarly, if   $\bar v := \min \, \big ( v_{\lambda}( x  - \lambda \, \tau), v_{\lambda'}(-x - \lambda' \, \tau ) \big )$
is the approximation of   $\bar u:=u_\blaaa$ with $\bar h = \bar h'=0$, then
$$\big | \int_{\R} \bar v^p (\cdot,\tau) \, dx  - \int_{\R} u^p (\cdot,\tau) \, dx \big | \leq  C_\blaaa\,  e^{d\, \tau}.$$
We conclude that
\be\label{eqn-v66}
\big | \int_{\R} \bar v^p (\cdot,\tau) \, dx  - \int_{\R} \bar v^p (\cdot,\tau) \, dx \big | \leq  C \,   e^{d\, \tau}
\ee 
We will now show that if $(h,h') \neq (0,0)$, then \eqref{eqn-v66} cannot hold leading to a contradiction. 
To this end, denote by $x(\tau)$ the intersection point between $v_{\lambda}(x + h -\lambda\tau)$ and $v_{\lambda'}(-x+h-\lambda'\tau)$
and by  $\bar{x}(\tau)$ the intersection point between $v_{\lambda}(x -\lambda\tau)$ and $v_{\lambda'}(- x  - \lambda'\tau )$.
We have
$$
\int_{\R}  v^p(\cdot,\tau)  \, dx = \int_{-\infty}^{x(\tau) +  h} v_{\lambda}^p(x  - \lambda\tau)\, dx + \int_{x(\tau)-  h'}^{+\infty} v_{\lambda'}^p(- x - \lambda'\tau)\, dy
$$
and similarly
$$
\int_{\R} \bar v^p(\cdot,\tau)  \, dx = \int_{-\infty}^{\bar{x}(\tau) } v_{\lambda}^p(x  - \lambda\tau)\, dx + \int_{\bar{x}(\tau)}^{+\infty} v_{\lambda'}^p(- x - \lambda'\tau)\, dx
$$
Hence,
$$
\int_{\R}  \bar  v^p(\cdot,\tau)  \, dx - \int_{\R}  v^p(\cdot,\tau)  \, dx =  \int_{x(\tau) +  h}^{\bar x(\tau)}  v_{\lambda}^p(x  - \lambda\tau)\, dx 
+ \int^{x(\tau)-  h'}_{\bar{x}(\tau)} v_{\lambda'}^p(- x - \lambda'\tau)\, dx. 
$$
By \eqref{eq-intersection} we have
$$
\bar{x}(\tau) - x(\tau) =  \frac{ h \gamma_{\lambda} - h' \gamma_\lambda'}{\gamma_{\lambda} + \gamma_\lambda'} + o(1)$$
which implies 
\begin{equation}
\label{eq-diff-inter}
\bar{x}(\tau) - x(\tau)  =
h - \frac{\gamma_{\lambda'}\, (h'+h)}{\gamma_{\lambda} + \gamma_\lambda'} + o(1) = -h'  + \frac{\gamma_{\lambda}\, (h'+h)}{\gamma_{\lambda} + \gamma_\lambda'} + o(1)
\end{equation}
Setting $\mu:=  \frac{\gamma_{\lambda}}{\gamma_{\lambda} + \gamma_\lambda'} $, $\mu':=  \frac{\gamma_{\lambda'}}{\gamma_{\lambda} + \gamma_\lambda'} $ and combining the above yields
\be\label{eqn-v77}
\begin{split}
\int_{\R}  \bar  v^p(\cdot,\tau)  \, dx &- \int_{\R}  v^p(\cdot,\tau)  \, dx \\&=  \int_{x(\tau) +  h}^{x(\tau)+h-\mu'(h+h')}  v_{\lambda}^p(x  - \lambda\tau)\, dx 
+ \int^{x(\tau)- h'}_{x(\tau) - h' + \mu (h+h')} v_{\lambda'}^p(- x - \lambda'\tau)\, dx  + o(1)
\end{split}
\ee
For  $h$, $h'$ and $\tau <<0$ (depending on $h$, $h'$) we have 
$v_{\lambda}(x-\lambda \tau)  \geq 1/2$ 
and $v_{\lambda'}(-x-\lambda'\tau)  \ge 1/2$ on the intervals over which those  functions are integrated in \eqref{eqn-v77}. In addition,
both integrals on the right hand side of \eqref{eqn-v77} have the same sign. Hence, 
$$\big | \int_{\R}  \bar  v^p(\cdot,\tau)  \, dx  - \int_{\R}  v^p(\cdot,\tau)  \, dx   \big |  \geq \frac 12 \, (\mu + \mu') \, (h+h') +o(1)= \frac 12 \, |h+h'| + o(1)$$
If $h+h'\neq 0$ this contradicts \eqref{eqn-v66} and concludes the proof of the Lemma. If $h'=-h$ then $v(x,\tau) = \bar v (x+h,\tau)$ for all $\tau$,
which means that the solutions $u_m, \bar u_m$ of \eqref{eqn-um} with initial data $v(\cdot,-m), \bar v(\cdot, -m)$ respectively satisfy 
$u_m(x,\tau) =\bar u_m(x+h,\tau)$ for all $\tau > -m$, hence the same will hold for the limits $u, \bar u$. Since $u=\bar u$, this means that
$u(x,\tau) =u(x+h,\tau)$ for any $x \in \R$. 
On the other hand, the fact that 
the solution $u$ defines  a metric that can be lifted to a smooth metric  on $S^n$ implies that $u(x, \tau) = C(\tau) \, e^{x} (1 + o(1))$, as $x \to -\infty$ with  $C(\tau) >0$,
hence $u(x,\tau) =u(x+h,\tau)$ must imply that $h=0$ which means that $(h,h')=(0,0)$ and contradicts our assumption. The proof of Step \ref{step-h} is now complete. 
\end{proof}

\begin{proof}[Proof of Theorem \ref{thm1}] The proof of the theorem is a direct consequence of Propositions \ref{prop-limit} and \ref{prop-distinguish}. 
\end{proof} 

\section{The geometry  of merging traveling waves} 
\label{sec-geom}

In this last section we derive the geometric properties of the ancient solution $u_\bla$ of the equation \eqref{eqn-v} on $\R \times (-\infty, T_\bla)$, as constructed 
in Theorem \ref{thm1}.  The one parameter family of metrics    $g_\bla(\tau):= u_\bla^\mmm(\cdot, \tau) \, g_{_{cyl}}$   can be lifted to a smooth one parameter family of metrics on $S^n \times (-\infty, T_\bla)$ which defines an ancient  rotationally symmetric  solution  of the  {\em rescaled Yamabe flow}  on $S^n$, equation 
\begin{equation}
\label{eq-RYF}
\frac{\partial}{\partial \tau} g = -(R - 1) \, g.
\end{equation}

\smallskip
We next prove the following result concerning the behavior of the Riemannian curvature  of the metric $g_\bla(\tau)$ near $\tau=  -\infty$. 

\begin{thm}\label{thm2}
The solution  $g_\bla(\tau):=u_\bla^\mmm(\cdot, \tau) \, g_{_{cyl}}$ defines a type I ancient solution to the Yamabe flow in the sense that the norm of its curvature operator  is uniformly bounded,  that is for any $\tau_0 < T_\bla$, we have 
$\|\,  { \mbox{\em Rm}} \, (g_\bla)\| \le C$  for all $ \tau \in (-\infty, \tau_0)$. 
\end{thm}  

\smallskip 

\bremark{ The statement of Theorem \ref{thm2}  exactly means that the unrescaled flow \eqref{eq-YF}, whose scaling by $|t|$ yields to the equation 
\eqref{eq-RYF}, is a type I ancient solution according to the Definition \ref{defn-ancient}}. 

\eremark

\begin{proof}
Since our metric is conformally flat, the norm of its curvature operator  $\|\mbox{\text  \rem}\|$, can be expressed in terms of the powers (positive or negative) of the conformal factor, its first and second order derivatives. On the other hand, the conformal factor satisfies the equation of type \eqref{eqn-v} in the considered parametrization. Therefore, we see that if we have  uniform upper and lower bounds on the conformal factor, by standard parabolic estimates we get uniform bounds on all its derivatives and therefore the uniform bound on $\|\rem\|$.

Estimate \eqref{eqn-qtau5} will be  crucial in proving this  theorem, that is we have 
\[\int_{\mathbb{R}} (v_\bla^p - u_\bla^p )\, dx \le C e^{d\tau}.\]
Denote shortly by $v := v_\bla$ and by $u := u_\bla$. Since $u \le v$, the previous estimate and the definition of $v:=
\min (u_{\la,h}, u_{\la',h'})$  imply  the bound 
\[\int_{-\infty}^{x(\tau)} (u_{\la,h}^p - u^p)\, dx \le C\, e^{d\tau}\]
where $u_{\la,h} = v_\la(x - \la\tau + h)$ is a traveling wave coming in from the left and  $x(\tau)$ given by \eqref{eq-intersection} denotes the point where
the two traveling waves $u_{\la,h}$ and $u_{\la',h'}$ intersect. 
Let $z = x - \la\tau + h$ and denote by $U(z,\tau) := u(z+\la\tau-h)$. If we perform this change of variables in the previous integral estimate we obtain,
\begin{equation}
\label{eq-vpup}
\int_{-\infty}^{\bar{x}(\tau)} (v_\la(z)^p - U(z,\tau)^p)\, dz \le C\, e^{d\tau}
\end{equation}
where $\bar{x}(\tau) := x(\tau) - \la\tau + h$ and $U(z,\tau)$ satisfies the equation
\begin{equation}
\label{eq-U-z}
U_{\tau} = U_{zz} + \la\, (U^p)_z - U + U^p.
\end{equation}
We will obtain derivative estimates   which hold for 
\be\label{eqn-bs}
- \infty < z < \bar x(\tau) + \frac 12
\ee
since similar estimates may be obtained in the region $\bar x(\tau) - \frac 12 < z < + \infty$ from the symmetry of our problem.

We go now from cylindrical to polar coordinates via the following coordinate change,
\begin{equation}
\label{eq-coord-change}
U(z,\tau) = \hat{u}(y,\tau)\, |y|^{\frac{2}{p-1}}, \qquad r = |y| = e^{\frac{p-1}{2}\, z}
\end{equation}
where $\hat{u}(y,\tau)$ satisfies the equation 
\begin{equation}
\label{eq-u-polar}
(\hat{u}^p)_{\tau} = \alpha \, \Delta \hat u + \beta\,  r \, \hat{u}_r + \gamma \, \hat{u},
\end{equation}
for some  constants $\alpha>0$ and $\beta, \gamma$. Furthermore, the ancient solution $g_\bla = \hat{u}^{\frac{4}{n-2}}\, g_{\mathbb{R}^n}$ has positive scalar curvature $R>0$, which is equivalent to 
\[\Delta_{\mathbb{R}^n} \hat{u} \le 0.\]
By the Mean value theorem we have
\begin{equation}
\label{eq-MVT}
\hat{u}(y_0,\tau) \ge C_n\, \int_{B(y_0,1)} \hat{u}(y,\tau)\, dy
\end{equation}
for all $y_0 \in \R^n$. 
Assume first that $|y| \le 2M$ for a fixed number $M$. Then, $u \le v$ implies
\begin{equation}
\label{eq-upp-bound}
\hat{u} \le \hat{v} = v_\la(z)\, |y|^{-\frac{2}{p-1}} \le C\, \frac{\min\{1, e^z\}}{|y|^{\frac{2}{p-1}}}  = C\, \frac{\min\{1, |y|^{\frac{2}{p-1}}\}}{|y|^{\frac{2}{p-1}}} \le C.
\end{equation}
Here we have used the estimate $v_\la (z)  \leq \min(1, e^z)$ which follows from the bounds $v_\la\leq 1$ and \eqref{eqn-vla}. 
Since $p >1$, \eqref{eq-MVT} and \eqref{eq-upp-bound}  imply
\begin{equation}
\label{eq-MVTp}
\hat{u}(y_0,\tau) \ge C\, \int_{B(y_0,1)} \hat{u}(y,\tau)^p\, dy.
\end{equation}
On the other hand, after the coordinate change \eqref{eq-coord-change}, estimate \eqref{eq-vpup} becomes
\[\int_{B(0,e^{\frac{p-1}{2}\, \bar{x}(\tau)})} \frac{(\hat{v}_\la^p - \hat{u}^p)}{|y|^{\frac n2 -1}}\, dy \le C e^{d\tau},\]
where $B\left(0,e^{\frac{p-1}{2}\, \bar{x}(\tau)}\right)$ is the euclidean ball in $\mathbb{R}^n$ of radius 
$e^{\frac{p-1}{2}\, \bar{x}(\tau)}$. Note that for $|y_0| \le 2M$, and $\tau < < -1$ sufficiently small so that $e^{\frac{p-1}{2}\, \bar{x}(\tau)} >> 1$, the previous estimate yields
\[c \, \int_{B(y_0,1)} (\hat{v}_\la^p - \hat{u}^p)\, dy \le \int_{B(0,e^{\frac{p-1}{2}\, \bar{x}(\tau)})} \frac{(\hat{v}_\la^p - \hat{u}^p)}{|y|^{\frac n2 - 1}}\, dy \le C\, e^{d\tau},\]
where $c = c(M)$ is a constant uniform in time. Hence,
\begin{equation}
\label{eq-comp-upvp}
\int_{B(y_0,1)} \hat{u}(y,\tau)^p\, dy \ge \int_{B(y_0,1)} \hat{v}_\la(y)^p \, dy - C e^{d\tau}, \qquad |y_0| \le 2M\end{equation}
where $C = C(M)$. Combining \eqref{eq-MVTp} and \eqref{eq-comp-upvp} yields
\[\hat{u}(y_0,\tau) \ge C\, \int_{B(y_0,1)} \hat{v}_\la(y)^p\, dy - C\, e^{d\tau} \ge c > 0\]
for $\tau \le \tau_0$ sufficiently small and all $|y| \le 2M$. This together with \eqref{eq-upp-bound} imply
\begin{equation}
\label{eq-nondeg}
c \le \hat{u}(y,\tau) \le C, \qquad \tau \le \tau_0, \qquad |y| \le 2M.
\end{equation}
Having \eqref{eq-nondeg}, equation \eqref{eq-u-polar} is a uniformly parabolic equation for $(y,\tau) \in B(0,2M)\times (-\infty,\tau_0)$, so standard parabolic estimates applied to equation \eqref{eq-u-polar} imply we have all uniform bounds on the derivatives of $\hat{u}$ in the region $B(0, \frac {3M}2) \times 
(-\infty, -2 \, \tau_0)$. Since $\hat{u}^{\frac{4}{n-2}}$ is the conformal factor of our metric $g_\bla$ in polar coordinates, by the discussion at the beginning of the proof we have
\[\| \,  { \mbox{\text  Rm}} \, (y,\tau)\| \le C, \qquad \tau \le \tau_0, \qquad |y| \le M\]
for a uniform constant $C$. Equivalently, in $z$ coordinates this means
\begin{equation}
\label{eq-curv-tip}
\| \,  { \mbox{\text Rm}} \, (z,\tau)\| \le C, \qquad \tau \le \tau_0, \qquad z \le \frac 2{p-1} \, \log M.
\end{equation}
Observe this estimate implies that we have  the curvature uniformly bounded in the tip region of our ancient solution.

Let us now focus on the inner part of our solution that turns out to have the asymptotics of a cylindrical metric. 
More precisely, we will assume now that $|y| \geq M/2$ which according to \eqref{eq-coord-change} means $$ z \ge \frac 2{p-1} \log \frac M2$$
and also that $z \leq \bar x(\tau) + 1$, since  we are interested in  deriving estimates in the region \eqref{eqn-bs}. 
 Recall that the estimate \eqref{eqn-qtau5} can be rewritten as
\[\int_{\mathbb{R}} \hat{a}\, (v - u)\, dx \le C\, e^{d\tau},\]
where $\hat{a} = p\, \int_0^1 (s v + (1-s) u)^{p-1}\, ds$. This implies
\[\int_{\la\tau - h + \frac 2{p-1} \,\log \frac M2}^{x(\tau) + 1} \hat{a}\, (v_\la(x-\la\tau+h) - u(x,\tau))\, dx \le C\, e^{d\tau}.\]
Let  $z = x -\la\tau + h$ be the coordinate change in the previous integral. Then,
\begin{equation}
\label{eq-middle}
\int_{\frac 2{p-1} \,\log \frac M2}^{\bar x(\tau) +1} a(z,\tau)\, (v_\la(z) - U(z,\tau))\, dz \le C\, e^{d\tau},
\end{equation}
where $a(z,\tau) := \hat{a}(z+\la \tau -h,\tau)$. Note that $a(z,\tau) \ge v_\la(z)^{p-1}$ and  we may choose $M \geq 2$ so that $\log M/2 \geq 0$,
hence 
\begin{equation}
\label{eq-a}
1 \ge a(z,\tau)^{\frac 1{p-1}}  \ge v_\la(z) \ge \frac 12, \qquad z \in (\frac 2{p-1} \,\log \frac M2, \bar  x(\tau))
\end{equation}
since $v_\la(z)$ increases in $z$ and $v_\la(0) = \frac 12$ by our normalization. Set $w(z,\tau) := v_\la(z) - U(z,\tau)$. Then $w \leq v_\la \leq 1$.
Hence, \eqref{eq-middle}  and \eqref{eq-a} imply that for any $q >1$, 
\begin{equation}
\label{eq-est-w}
\int_{\frac 2{p-1} \,\log \frac M2}^{\bar{x}(\tau)+1} w(z,\tau)\, dz \le C e^{d\tau}. 
\end{equation}
On the other hand, since  both  $U(z,\tau)$ and $v_\la(z)$ satisfy equation \eqref{eq-U-z} we get that $w(z,\tau)$ satisfies
\[(a\,w)_\tau = w_{zz} +\la w_z - w + w\, a\]
and by  \eqref{eq-a} the  equation is uniformly parabolic. Hence, standard parabolic estimates applied to it and estimate \eqref{eq-est-w} yield 
the $C^k$ bound 
\[\|w\|_{C^k(U_\tau)} \le C_k e^{d\tau},\]
where $U_\tau = [ \frac 2{p-1} \log 3M/4, \bar{x}(\tau)) \times (-\infty, 2\tau_0)$ and $\tau_0 << -1$ is sufficiently small. 
In particular, in the considered region, we have
\[U(z,\tau) \ge v_\la(z) - C\, e^{d\tau} \ge c > 0, \qquad \mbox{in} \qquad U_\tau.\]
This implies the bound 
\begin{equation}
\label{eq-curv-cylind}
\| \,  { \mbox{\text Rm}} (z,\tau)\| \le C, \qquad \tau \le 2\tau_0, \qquad  \frac 2{p-1} \log M \leq z \leq  \bar{x}(\tau).
\end{equation}

Finally, estimates \eqref{eq-curv-tip} and \eqref{eq-curv-cylind} yield a desired uniform bound on $\| {\text Rm} \|$ for our ancient solution $g_\bla$ for all $\tau \le 2\tau_0$ and all $x \le x(\tau)$. Recall that for $x \ge x(\tau)$ we get the uniform curvature bound using the same analysis as above (the only difference is that this time we need to consider the soliton that is coming in from the right). This finishes the proof of the Theorem.
\end{proof}

\begin{remark}\label{prop-sectional} {\em Let $u_\bla$ be the ancient solution to \eqref{eqn-v} as in Theorem \ref{thm1}. 
Then, we  will next observe that   the metric $g_\bla:= (u_\bla)^\mmm \, g_{cyl}$ has nonnegative  Ricci curvature. 
Indeed, recall that $u_\bla = \lim_{m \to +\infty} u_m$, where $u_m$ is the solution of the initial value problem \eqref{eqn-um}.
It is  sufficient to see  that each $u_m$ has nonnegative  Ricci curvature, since  then we can pass to the limit $m \to +\infty$.
Indeed,  we have seen that  the convergence of $\{ u_m \}$
to $u_\bla$  is uniform on compact subsets of $\R \times (\infty, +\infty)$ and  that 
$u_\bla >0$ on $\R \times (-\infty,T_\bla)$, where  $T_\bla$ is  uniform in $m$. Standard regularity arguments on the quasilinear  equation
\eqref{eqn-v} imply that the convergence is $C^\infty$ on compact subsets of $\R \times (\infty, T_\bla)$, from which our claim readily follows. 

We will next observe how one may show  that each solution $g_m(\cdot, \tau) := u_m(\cdot, \tau)\, g_{cyl}$ of \eqref{eqn-um} has nonnegative Ricci curvature.  
The initial data of $u_m$ at $\tau = -m$  is $v_\bla(\cdot, -m)$. 
We recall that for every  $\tau$, we have defined $v_\bla$ by \eqref{eqn-va2},  namely  
$v_{\la, \la', h, h'}(\cdot, \tau) =  \min \big ( v_\la(x- \la \tau  +h), v_{\la'}(-x- \la' \tau+h')  \big )$ where $v_\la$, $v_{\la'}$ are
traveling wave solutions of equation \eqref{eqn-v2}. It has been shown in \cite{DKS} (Section 4, Proposition 4.5) that both metrics defined via conformal factors $v_\la$, $v_{\la'}$ respectively, have
nonnegative  sectional curvatures. 

Moreover, it  has been observed  in \cite{DKS}  (Section 4) that for  a given smooth and rotationally symmetric metric $g:= v(x)   \, g_{cyl} $ where 
 $g_{cyl} := dx^2 + g_{S^{n-1}}$,  nonnegative   sectional curvatures  is equivalent  to having 
\be\label{eqn-sec1}
 v_x^2 - v\, v_{xx} \geq 0 \qquad \mbox{and} \qquad 4 v^2 - v_x^2 \geq  0.
 \ee 
Since each  for each $\tau \in \R$, the functions $v_\la(\cdot,\tau), v_{\la'}(\cdot,\tau)$ satisfy \eqref{eqn-sec1} (up to the dilation performed 
in \eqref{eqn-cv2}),  the minimum $v_\bla(\cdot,\tau)$
also satisfies \eqref{eqn-sec1} (up to the same dilation) in the distributional sense and it is smooth on $\R \setminus \{ x(\tau) \}$,
where $x(\tau)$ denotes the point at which $v_\la(\cdot,\tau)$ and $v_{\la'}(\cdot,\tau)$  intersect. 
One can then show that there is an approximation $\{ v_\bla^\delta (\cdot,\tau)\}$, $\delta \in (0,\delta_0)$  of 
$v_\bla(\cdot, \tau)$ each satisfying \eqref{eqn-sec1} and such that $ v_\bla^\delta  (\cdot,\tau) \to v_\bla (\cdot,\tau)$, as $\delta \to 0$ uniformly on
compact subsets in $\R$ and also in $C^\infty$ on compact subsets of $\R \setminus \{ x(\tau) \}$. 

Let $u_m^\delta$ be the solution to \eqref{eqn-um} with initial data $v_\bla^\delta(\cdot, -m)$ instead of $v_\bla(\cdot, -m)$.
Since $g_m^\delta(\cdot, -m) := (v_\bla^\delta)^\mmm(\cdot, -m)\, g_{cyl}$ has nonnegative  sectional curvatures, it also has nonnegative
Ricci curvature and this is preserved by the Yamabe flow. It follows that  $g_m^\delta(\cdot, \tau) := (u_m^\delta)^\mmm(\cdot, \tau)\, g_{cyl}$
has nonnegative curvature and passing to the limit $\delta \to 0$, the same holds for $g_m(\cdot, \tau) := (u_m)^\mmm(\cdot, \tau)\, g_{cyl}$.
This sketches the proof of the claim about our solutions having nonnegative Ricci curvature.}
\end{remark}

\medskip

\begin{proof}[Proof of Theorem \ref{thm}]  Theorem \ref{thm} follows as a  direct consequence of  Theorems \ref{thm1} and \ref{thm2}. 

\end{proof}

\centerline{\bf Acknowledgements}

\medskip

\noindent P. Daskalopoulos  has been partially supported by NSF grant DMS-1266172.\\
M. del Pino  has been supported by grants Fondecyt 1110181 and Fondo Basal CMM.\\
N. Sesum has been partially supported by NSF grant  DMS-1056387.

\end{document}